\documentclass[a4paper,12pt]{amsart}
\usepackage{geometry}

\usepackage{amsmath,amsthm,amsfonts,amssymb,amscd}
\usepackage{fullpage}
\usepackage{lastpage}
\usepackage{enumerate}
\usepackage{fancyhdr}

\usepackage[american]{babel}

\usepackage{mathrsfs}
\usepackage{latexsym,epsfig,hyperref}          %allows figures to be inserted in .eps format

\usepackage{pstricks}

\usepackage{enumitem}% enumerate Roman
\usepackage{palatino}

\usepackage{tikz-cd}

\newtheorem{thm}{Theorem}
\newtheorem{remark}[thm]{Remark}
\newtheorem{prop}[thm]{Proposition}

\newtheorem{lemma}[thm]{Lemma}

\newtheorem*{theoremM}{Main Theorem}

\def\dim{\operatorname{dim}}

\def\Pic{\operatorname{Pic}}%
\def\Aut{\operatorname{Aut}}%
\def\Div{\operatorname{Div}}%
\def\Spec{\operatorname{Spec}}%
\def\Proj{\operatorname{Proj}}%
\def\Ext{\operatorname{Ext}}%

\title{ Non-reduced components of the Hilbert scheme of curves \\ using triple covers }

\author[Y. Choi]{Youngook Choi 
}
\address{Department of Mathematics Education, Yeungnam University, 280 Daehak-Ro, \hfill \newline\texttt{}
 \indent Gyeongsan, Gyeongbuk 38541, Republic of Korea}
\email{ychoi824@yu.ac.kr}

\author[H. Iliev]{Hristo Iliev 
}
\address{American University in Bulgaria, 2700 Blagoevgrad, Bulgaria, and \hfill \newline\texttt{}
 \indent Institute of Mathematics and Informatics, Bulgarian Academy of Sciences, \hfill \newline\texttt{}
 \indent 1113 Sofia, Bulgaria}
\email{ hiliev@aubg.edu, hki@math.bas.bg}

\author[S. Kim]{Seonja Kim 
}
\address{Department of Electronic Engineering, Chungwoon University, Sukgol-ro, Nam-gu, \hfill \newline\texttt{}
 \indent Incheon 22100, Republic of Korea}
\email{sjkim@chungwoon.ac.kr}

\thanks{The first author was supported by Basic Science Research Program through the National Research Foundation of Korea(NRF)
funded by the Ministry of Education(2019R1I1A3A01055643).
The second author was supported by Grant KP-06-N 62/5 of Bulgarian National Science Fund.
The third  author was supported by the National Research Foundation of Korea(NRF) grant funded by the Korea government(MSIT) (2022R1A2C1005977).}

\makeatletter
\@namedef{subjclassname@2020}{\textup{2020} Mathematics Subject Classification}
\makeatother

\usepackage{setspace}
\begin{document}

\setlength{\parindent}{5ex}

\begin{abstract}
In this paper we consider curves on a cone that pass through the vertex and are also triple covers of  the base of the cone, which is a general smooth curve of genus $\gamma$ and degree $e$ in $\mathbb{P}^{e-\gamma}$. Using the free resolution of the ideal of such a curve found by Catalisano and Gimigliano, and a technique concerning deformations of curves introduced by Ciliberto, we show that the deformations of such curves remain on cones over a deformation of the base curve. This allows us to prove that for $\gamma \geq 3$ and $e \geq 4\gamma + 5$ there exists a non-reduced component $\mathcal{H}$ of the Hilbert scheme of smooth curves of genus $3e + 3\gamma$ and degree $3e+1$ in $\mathbb{P}^{e-\gamma+1}$. We show that $\dim T_{[X]} \mathcal{H} = \dim \mathcal{H} + 1 = (e - \gamma + 1)^2 + 7e + 5$ for a general point $[X] \in \mathcal{H}$.
\end{abstract}


\maketitle

\section{Introduction}\label{Sec1}

Let $\mathcal{I}_{d,g,r}$  denote the union of irreducible components of the Hilbert scheme whose general points correspond to smooth irreducible non-degenerate curves of degree $d$ and genus $g$ in $\mathbb{P}^r$. The minimal possible dimension that a component of $\mathcal{I}_{d,g,r}$ can have is $\lambda_{d,g,r} := (r+1)d - (r-3)(g-1)$. Recall that this number is called the \emph{expected dimension}. Note that $\lambda_{d,g,r} = h^0 (C, N_{C/\mathbb{P}^r}) - h^1 (C, N_{C/\mathbb{P}^r})$, where $N_{C/\mathbb{P}^r}$ is the normal bundle to a curve $C$ represented by a point of $\mathcal{I}_{d,g,r}$. For the tangent space to a component $\mathcal{H}$ of $\mathcal{I}_{d,g,r}$ at a point $[C] \in \mathcal{H}$ we have
\[
 \dim \mathcal{H} \leq \dim T_{[C]} \mathcal{H} = h^0 (C, N_{C/\mathbb{P}^r}) \, .
\]
If for a general $[C] \in \mathcal{H}$ we have equality, then the component is generically smooth. Whenever $\dim \mathcal{H} < \dim T_{[C]} \mathcal{H} $ at a general $[C] \in \mathcal{H}$, the component is non-reduced.

After obtaining two series of components of the Hilbert scheme $\mathcal{I}_{d,g,r}$ of curves in \cite[Theorem 4.3, Theorem 4.4]{CIK17}, we realized that the curves parametrized by them are found on cones, which allowed us to compute $h^0 (X, N_{C/\mathbb{P}^r})$ for a curve $C$ represented by a general point of such component. In this way we were able to strengthen some of the results proved in \cite{CIK17} and to describe the corresponding curves in a more geometric fashion. In \cite[Theorem A]{CIK21} we identified a series of generically smooth components of $\mathcal{I}_{2g-4\gamma+2, g, r}$ for every $\gamma \geq 10$ and $\gamma \leq r \leq g - 3\gamma + 2$, which extended \cite[Theorem 4.3]{CIK17}. In our paper \cite{CIK22} we found a series of non-reduced components of $\mathcal{I}_{2g-4\gamma+1, g, g-3\gamma+1}$ for every $\gamma \geq 7$ and $g \geq 6\gamma + 5$. We proved that the corresponding non-reduced components parametrize curves that lie on cones, pass through the vertex of the corresponding cone and are double covers of its general hyperplane section, which is a linearly normal nonspecial curve of genus $\gamma$. We remark that the non-reduced components from \cite{CIK22} are related to those in \cite[Theorem 4.4]{CIK17}.

In the present work we continue our study of smooth curves on cones that pass through the vertex of a cone and are $m$-covers, $m \geq 3$, of the hyperplane section of the cone. The main result in the paper concerns the case $m = 3$ and says that under suitable numerical assumptions such families of curves give rise to non-reduced components of the Hilbert scheme of curves.  It is formulated in the next theorem.

\begin{theoremM}\label{TheoremB3}
Assume that $e$ and $\gamma$ are integers such that $e \geq 4 \gamma +5$ and $\gamma \geq 3$. Let
\[
  g := 3\gamma + 3e \, , \qquad  d:= 3e + 1 \quad \mbox{ and } \quad r:= e - \gamma + 1 \, .
\]
Then the Hilbert scheme $\mathcal{I}_{d, g, r}$ possesses a \emph{non-reduced} component $\mathcal{H}$ such that
\begin{enumerate}[label=(\roman*), leftmargin=*, font=\rmfamily]
  \item $\dim \mathcal{H} = r^2 + 7e + 4$;

  \item at a general point $[X] \in \mathcal{H}$ we have $\dim T_{[X]} \mathcal{H} = \dim \mathcal{H} + 1$;

  \item a general point $[X] \in \mathcal{H}$ represents a curve $X$ lying on a cone $F$ over a smooth curve $Y$ of genus $\gamma$ and degree $e$ in $\mathbb{P}^{r-1}$ such that
  \begin{enumerate}[label=\theenumi.\arabic*.]
   \item $X \subset \mathbb{P}^{r}$ is projectively normal and passes through the vertex $P$ of the cone $F$;

   \item there is a line $l$ from the ruling of $F$ that is tangent to $X$ at $P$ as the intersection multiplicity is two;

   \item the projection from $P$ to the hyperplane in $\mathbb{P}^{r}$ containing the curve $Y$ induces a morphism $\varphi : X \to Y$ of degree three ;

   \item the ramification divisor $R_{\varphi}$ is linearly equivalent to the divisor cut on $X$ by a quadric hypersurface together with $Q_1 + Q_2$, where $Q_1$ and $Q_2$ are the remaining two points in which the tangent line $l$ intersects $X$ besides $P$.
  \end{enumerate}
\end{enumerate}
\end{theoremM}

The above result implies that the structure sheaf of the Hilbert scheme consists of commutative algebras with nonzero nilpotent elements. The first example of a nonreduced component of the Hilbert scheme of curves was produced by Mumford for $\mathcal{I}_{14,24,3}$ using space curves contained in cubic surfaces, see [Mum62]. It was generalized in Kleppe's systematic study [Kle87] of $s$-maximal families of curves in $\mathbb{P}^3$ for $s=3$, and subsequently by Kleppe and Ottem in [KO15] and [Kle2017] for $s = 4, 5$. Dan combined the analysis of the flag Hilbert schemes with the theory of Hodge loci to produce examples of nonreduced components with curves lying on surfaces in $\mathbb{P}^3$ of arbitrary degree, see [Dan17]. Interesting examples of nonreduced components of $\mathcal{I}_{d,g,3}$ have been given also in [GP82, Ell87, MP96, Nas06], see [Ser06, IV.6.2] for additional discussion. What all these examples have in common is their reliance on curves situated on surfaces in $\mathbb{P}^3$. Consequently, to construct nonreduced components of $\mathcal{I}_{d, g, r}$ for $r \geq 4$, an alternative approach is necessary. This task appears more challenging for $r \geq 4$, with the sole known example in this scenario being the one presented by Ciliberto, Lopez and Miranda in  [CLM96, Theorem 4.11].

Although our main result fits in the context of \cite{CIK21}, \cite{CIK22} and \cite{FS21}, it is independent of them. To obtain it, we develop the approach used in \cite{CIK22}, use the characterization of smooth curves on a cone that pass through its vertex given in \cite{CG99}, and apply similar arguments to those used in \cite{CLM96} and \cite{Cil87} to deduce that every deformation of a curve from the family of curves constructed in the theorem yields a curve from the same family. We remark that the technique used in the proof of our {\rm Main Theorem} cannot be by applied in the proof of \cite[Theorem B]{CIK22}, as we explain the reasons for this in {\rm Remark \ref{Sec3Remark2-3case}}. On the other hand, the possibility for curves on cones, which are algebraically equivalent to a  high degree hypersurface intersection plus a line, to yield a non-reduced component of the Hilbert scheme of curves has already been suggested in \cite[Remark 4.12]{CLM96}. In this sense our work was inspired by \cite{CLM96}.

The free resolution of the ideal of a smooth curve on a cone passing through its vertex, obtained by Catalisano and Gimigliano in \cite{CG99}, plays an essential role in the proof of our main result. For this reason we describe their result in section \ref{Sec2} using a setup that fits the framework of the Main Theorem. Further in the same section we prove several results about smooth curves on cones that are $m:1$ covers of the base of the cone and pass through its vertex. Also, for $m = 3$ we prove a technical result, namely {\rm Proposition \ref{Sec2PropPushForwardStrSheaf}}, that plays an important role in the proof of our Main Theorem, which is given in section \ref{Sec3}.

We work over the field $\mathbb{C}$. By \emph{curve} we understand a smooth integral projective algebraic curve. Given a line bundle $L$ on a smooth projective variety $X$, or a divisor $\Delta$ associated to $L$, we denote by $|L|$ or $|\Delta|$ the complete  linear series  $\mathbb P\left(H^0(X,L)\right)$ on $X$. For a line bundle $L$ and a divisor $\Delta$ on a variety $X$, we abbreviate, occasionally, the notation of the line bundle $L \otimes \mathcal{O}_X (\Delta)$ to simply $L(\Delta)$.  We use $\sim$ to denote linear equivalence of divisors. Given a finite morphism $\varphi : X \to Y$ of curves and a divisor $\Delta = \sum n_i P_i$ on $X$, we denote by $\varphi ( \Delta )$ the divisor $\sum n_i \varphi (P_i)$ on $Y$.
When $X$ is an object of a family, we denote by  $[X]$ the corresponding point of the Hilbert scheme representing the family. For all other definitions and properties of objects not explicitly introduced in the paper the reader can refer to \cite{Hart77} and \cite{ACGH}.

\section{ Preliminary results }\label{Sec2}

In our paper \cite{CIK22} we constructed a series of non-reduced components of the Hilbert scheme of curves using curves that lie on cones as each curve passes through the vertex of the corresponding cone. There, we considered only curves that are double covers of the base of the cone. On the other hand, curves on cones that are $m:1$ covers of the base, $m \geq 2$, and pass through the vertex have been studied by Catalisano and Gimigliano in \cite{CG99} with a different aim. Motivated by an earlier work of Jaffe about smooth curves on a cone that pass through its vertex, see \cite{Jaf91}, Catalisano and Gimigliano showed in \cite{CG99} that such curves are projectively normal, provided that the base curve of cone is, and gave a resolution of the ideal of such a curve in terms of a resolution of the ideal of the base curve. We will formulate below the main result of \cite{CG99}. For this assume that:

\begin{itemize}[font=\sffamily, leftmargin=1.0cm,
style=nextline]
 \item[$\Gamma$] is a smooth integral curve of genus $\gamma$,

 \item[$E$] is a divisor of degree $e \geq 2\gamma + 1$ on $\Gamma$,

 \item[$q$] is a point on $\Gamma$,

 \item[$S$] is the ruled surface $S = \mathbb{P} (\mathcal{O}_{\Gamma} \oplus \mathcal{O}_{\Gamma} (-E))$,

 \item[$f$] is the natural projection morphism $f : S \to \Gamma$,

 \item[$\Gamma_0$] is the section of minimal self-intersection of $f : S \to \Gamma$, that is, the one that corresponds to the exact sequence
 \[
  0 \to \mathcal{O}_{\Gamma} \to \mathcal{O}_{\Gamma} \oplus \mathcal{O}_{\Gamma} (-E) \to \mathcal{O}_{\Gamma} (-E) \to 0
 \]
 with $\Gamma_0^2 = \deg \mathcal{O}_{\Gamma} (-E) = -e$,

 \item[$\Psi$] is the morphism determined by the linear series $|\Gamma_0 + E\mathfrak{f}|$ on $S$.
\end{itemize}

\noindent We remark that $\Psi$ is isomorphism away from $\Gamma_0$ and contracts $\Gamma_0$ to a point, see \cite{FP05NM} for more details. Thus, $\Psi$ maps $S$ into a cone, so we denote by
\begin{itemize}[font=\sffamily, leftmargin=1.0cm,
style=nextline]
 \item[$F$] the image of $S$ under $\Psi$, that is, $F = \Psi (S)$, and

 \item[$P$] the vertex  of the cone $F$, that is, $P = \Psi (\Gamma_0)$.
\end{itemize}

\noindent Set $r := \dim |\Gamma_0 + E\mathfrak{f}|$. Then the embedding $F \subset \mathbb{P}^r$ is complete and the hyperplane sections of $F$ are the images, under $\Psi$, of the divisors from the linear series $|\Gamma_0 + E\mathfrak{f}|$ on $S$. Let
\begin{itemize}[font=\sffamily, leftmargin=1.0cm,
style=nextline]
 \item[$\sigma_D$] be a section of $f : S \to \Gamma$ for whose image $\sigma_D (\Gamma) =: D$ we have that

 \item[$D$] is a smooth curve in the linear series $|\Gamma_0 + E \mathfrak{f}|$ on $S$, and let

 \item[$Y$] be the image of $D$ under $\Psi$, that is, $Y = \Psi (D)$.
\end{itemize}
\noindent The curves $\Gamma$, $D$ and $Y$ are isomorphic to one another since $\Psi$ is an isomorphism away from $\Gamma_0$ and $D \cdot \Gamma_0 = (\Gamma_0 + E \mathfrak{f}) \cdot \Gamma_0 = 0$. Also, by \cite[Proposition 1]{CIK21}, $r = e-\gamma+1$, and $Y$ is a smooth, linearly normal curve of genus $\gamma$ and degree $e$ in $\mathbb{P}^{r-1}$. In fact, due to $e \geq 2\gamma + 1$, it follows by \cite{Mum1969} that $Y$ is projectively normal. Thus, we can consider $F$ as a cone in $\mathbb{P}^r$ over the projectively normal curve $Y \subset \mathbb{P}^{r-1}$.

\begin{enumerate}[label=\normalfont(\alph*), leftmargin=*, font=\rmfamily]
 \item[(MS)] \emph{We call the above assortment of assumptions about $\Gamma$, $E$, $q$, $S$, $f$, $\Gamma_0$, $\Psi$, $F$, $P$, $r$, $D$, $\sigma_D$ and $Y$, and the properties we described, the \emph{Main Setup}, and we abbreviate it as {\rm (MS)}}.
\end{enumerate}

Catalisano-Gimigliano's result can now be formulated as follows.

\begin{prop}(\cite[Proposition 2]{CG99})\label{Sec2PropCatGim}
Assume the conditions and notations of {\rm (MS)}. Let  $C_m \in |m\Gamma_0 + (mE + q) \mathfrak{f}|$ be general and $X_m = \Psi (C_m)$ be the image of $C_m$ on $F$, where $m \geq 2$ is an integer. Then
\begin{enumerate}[label=\normalfont(\alph*), leftmargin=*, font=\rmfamily]
 \item $X_m$ is a smooth integral projectively normal curve that passes through the vertex $P$;

 \item given a free resolution of the ideal sheaf $\mathcal{I}_{Y}$ of $Y$
\begin{equation}\label{Sec2ResolY}
 0 \to \mathcal{F}_{r-2} \to \mathcal{F}_{r-3} \to \cdots \to \mathcal{F}_{1} \to \mathcal{I}_Y \to 0 \,
\end{equation}
with $\mathcal{F}_{i} = \bigoplus\limits^{\beta_i}_{j=1} \mathcal{O}_{\mathbb{P}^r} (-\beta_{i,j})$, $i = 1, \ldots , r-2$, the ideal sheaf $\mathcal{I}_{X_m}$  of $X_m$ has a free resolution
\begin{equation}\label{Sec2ResolXm}
 0 \to \mathcal{P}_{r-1} \to \mathcal{P}_{r-2} \to \cdots \to \mathcal{P}_{1} \to \mathcal{I}_{X_m} \to 0 \, ,
\end{equation}
where
\begin{itemize}[leftmargin=*, font=\rmfamily]
 \item[] $\mathcal{P}_{1} = \bigoplus\limits^{r-1}_{1} \mathcal{O}_{\mathbb{P}^r} (-m-1) \oplus \bigoplus\limits^{\beta_1}_{j=1} \mathcal{O}_{\mathbb{P}^{r}} (-\beta_{1,j})$

 \item[] $\mathcal{P}_k =  \bigoplus\limits^{\binom{r-1}{k}}_1 \mathcal{O}_{\mathbb{P}^{r}}(-m-k) \oplus \bigoplus\limits^{\beta_k}_{j=1} \mathcal{O}_{\mathbb{P}^{r}} (-\beta_{k,j}) \oplus \bigoplus\limits^{\beta_{k-1}}_1 \mathcal{O}_{\mathbb{P}^{r}}(-m-\beta_{k-1,j})$, for $2 \leq k \leq r-2$

 \item[] $\mathcal{P}_{r-1} = \mathcal{O}_{\mathbb{P}^{r}}(-m-r+1) \oplus \bigoplus\limits^{\beta_{r-2}}_1 \mathcal{O}_{\mathbb{P}^{r}}(-m-\beta_{r-2,j})$.
\end{itemize}
\end{enumerate}
\end{prop}

\begin{remark}\label{Sec2RemarkCatGim}
For any point $z \in \Gamma$ the morphism $\Psi$ maps the fiber $z\mathfrak{f}$ to a line from the ruling of $F$ passing through the point $\Psi (\sigma_D (z))$ on $Y$. Let $l_q \subset F$ be the line corresponding to $q$.
As it is pointed out in \cite[section 1]{CG99}, the curve $X_m$, together with $(e-1)$ lines $L_1, \ldots , L_{e-1}$ from the ruling of $F$, is cut on $F$ by a degree $(m+1)$ hypersurface $G_{m+1} \subset \mathbb{P}^r$, where $L_1, \ldots , L_{e-1}$ are the residual lines on $F$ cut by a hyperplane that contains the line $l_q$. We remark also that the smoothness of a general $C_m \in |m\Gamma_0 + (mE+q)\mathfrak{f}|$ follows by \cite{Jaf91} and \cite{CG99}.
\end{remark}

Note that since the curve $C_m$ is in linear equivalence class of $m\Gamma_0 + (mE + q)\mathfrak{f}$, the adjunction formula gives about its genus $g$
\[
 \begin{aligned}
  2g-2 & = (-2\Gamma_0 + (K_{\Gamma} - E)\mathfrak{f} + m\Gamma_0 + (mE + q)\mathfrak{f}) \cdot (m\Gamma_0 + (mE + q)\mathfrak{f}) \\
  & = m(m-1)e + 2m\gamma - 2 \, ,
 \end{aligned}
\]
hence $g = \binom{m}{2}e+m\gamma$. Likewise, $(\Gamma_0 + E\mathfrak{f}) \cdot C_m = me + 1$, so $X_m$ is a smooth curve of degree $d = me + 1$ and same genus $g$. We remark also that if $q_0$ is the point in which the fiber $q \mathfrak{f}$ meets $\Gamma_0$, then it follows by
\cite[Proposition 36]{FP05NM} that the linear series $|m\Gamma_0 + (mE + q)\mathfrak{f}|$ has a unique base point at $q_0$. This allows us to make the following observation about $X_m$.

\begin{prop}\label{SecPropLine}
Assume the conditions and notations of {\rm (MS)}. Let $l_q$ be as in {\rm Remark \ref{Sec2RemarkCatGim}} and  $X_m$ be as above. The line $l_q$ is tangent to $X_m$ at the point $P$ as their intersection multiplicity at $P$ is exactly two.
\end{prop}
\begin{proof}
The morphism $\Psi : S \to F$ is in fact the  blow-up of $F$ at the vertex $P$. Since $C_m$ is the proper transform of $X_m$ and $q\mathfrak{f}$ is the proper transform of $l_q$, they wouldn't meet on $\Gamma_0 = \Psi^{-1} (P)$ unless the intersection of $X_m$ and $l_q$ at $P$ is of multiplicity at least two. On the other hand $C_m \in |m\Gamma_0 + (mE + q)\mathfrak{f}|$ is general, hence $q\mathfrak{f}$ meets $C_m$ in additional $m-1$ points, all of which are distinct and away from $\Gamma_0$. Since $\Psi$ is an isomorphism away from $\Gamma_0$, the images of those $m-1$ points will be distinct points on $l_q$ and away from $P$. The inner projection with center $P$ to the hyperplane containing $Y$ yields an $m:1$ covering $X_m \to Y$, therefore the intersection multiplicity of $X_m$ and $l_q$ at $P$ can only be two.
\end{proof}

It is convenient to have an explicit notation for the morphism mentioned in the proof of the lemma, so denote by $\varphi : X_m \to Y$ the $m:1$ covering morphism induced by the inner projection with center $P$ to the hyperplane containing the curve $Y$. We remark that the image $\varphi (P)$ of the point $P$ is by definition the point in which the tangent line $l_q$ to $X_m$ at $P$ meets the hyperplane, which is the point $\Psi (\sigma_D (q)) =: Q$. Consider also the morphism $\phi : C_m \to D$ defined as the composition $\phi := \sigma_D \circ (f_{|_{C_m}})$. Note that the morphism $\phi$ coincides with the composition $(\Psi^{-1})_{|_Y} \circ \varphi \circ (\Psi_{|_{C_m}})$. Next we will derive a few facts involving the ramification divisor of $\varphi$ but before that we summarize, for convenience of the reader, the \emph{additional notations}. We will refer to them as  {\rm (AN)}.

\begin{itemize}[font=\sffamily, leftmargin=1.0cm, style=nextline]
 \item[(AN)]

\begin{itemize}[font=\sffamily, leftmargin=1.0cm, style=nextline]
 \item[$C_m$] is a general curve in the linear series $|m\Gamma_0 + (mE + q)\mathfrak{f}|$,

 \item[$q_0$\phantom{.}] is the unique base point of $|m\Gamma_0 + (mE + q)\mathfrak{f}|$; note that $q_0 \in \Gamma_0$,

 \item[$X_m$] is the image $\Psi (C_m) \subset F \subset \mathbb{P}^r$ of $C_m$, which is smooth according to {\rm Remark \ref{Sec2RemarkCatGim}},

 \item[$\varphi$\phantom{.}] is the $m:1$ covering morphism $\varphi : X_m \to Y$ induced by the projection with center $P$ to the hyperplane in $\mathbb{P}^r$ containing $Y$,

 \item[$\phi$\phantom{.}] is the $m:1$ covering morphism $\phi : C_m \to D$ defined as $\phi := \sigma_D \circ (f_{|_{C_m}})$,

 \item[$Q$\phantom{.}] is the point on $Y$ defined as the image $\Psi (\sigma_D (q))$ of $q \in \Gamma$.
\end{itemize}
\end{itemize}

\begin{prop}\label{Sec2PropRamBranch}
Assume the conditions and notations {\rm (MS)} and {\rm (AN)}. Denote by $R_{\varphi}$ the ramification divisor of the morphism $\varphi$. Then
\begin{enumerate}[label=\normalfont(\alph*), leftmargin=*, font=\rmfamily]
 \item $R_{\varphi}$ is linearly equivalent to the divisor cut on $X_m$ by a hypersurface of degree $(m-1)$ together with the $(m-1)$ points, besides $P$, in which the line $l_q$ meets $X_m$  ;

 \item $\deg R_{\varphi} = (m-1)(me + 2)$  ;

  \item for the branch divisor $\varphi (R_{\varphi})$ of $\varphi$ we have
$
  \mathcal{O}_Y (\varphi (R_{\varphi})) \cong \mathcal{O}_Y (m(m-1)) \otimes \mathcal{O}_Y (2(m-1)Q)
$;

 \item $\varphi^{\ast} \mathcal{O}_Y (1) \cong \mathcal{O}_{X_m} (1) \otimes \mathcal{O}_{X_m} (-P) $  .
\end{enumerate}
\end{prop}
\begin{proof}
Since $C_m$ and $X_m$ are isomorphic, we can transform some of the claims about $R_{\varphi}$ into claims about the ramification divisor $R_{\phi}$ of the morphism $\phi$, which are easier to prove.
\begin{enumerate}[label=(\alph*), leftmargin=*, font=\rmfamily]
 \item For $R_{\phi}$ we have
 \[
  \begin{aligned}
   R_{\phi} & \sim K_{C_m} - \phi^{\ast} K_D \\
   & \sim (-2\Gamma_0 + (K_{\Gamma}-E)\mathfrak{f} + m\Gamma_0 + (mE + q)\mathfrak{f})_{|_{C_m}} - \phi^{\ast} K_D \\
   & \sim ((m-2)\Gamma_0 + ((m-1) E + q)\mathfrak{f})_{|_{C_m}} + K_{\Gamma} \mathfrak{f}_{|_{C_m}} - \phi^{\ast} K_D \, .
  \end{aligned}
 \]
The divisor $K_D$ is the restriction of $K_S + D \sim -2\Gamma_0 + (K_{\Gamma}-E)\mathfrak{f} + \Gamma_0 + E \mathfrak{f} \sim -\Gamma_0 + K_{\Gamma}\mathfrak{f}$ on the curve $D$. However, $D$ doesn't meet $\Gamma_0$, therefore $\phi^{\ast} K_D \sim \phi^{\ast} ((-\Gamma_0 + K_{\Gamma}\mathfrak{f})_{|_D}) = \phi^{\ast} (K_{\Gamma}\mathfrak{f}_{|_D}) \sim K_{\Gamma} \mathfrak{f}_{|_{C_m}}$. Therefore
\[
 R_{\phi} \sim ((m-2)\Gamma_0 + ((m-1) E + q)\mathfrak{f})_{|_{C_m}} \, .
\]
By the commutativity of the diagram
\begin{equation}\label{Sec2PropCommutDiag}
    \begin{tikzcd}%[cells={anchor=west}]
      C_m \ar[r, "{\Psi_{|_{C_m}}}", "{\cong}"'] \ar[d, "{\phi}"'] & X_m \ar[d, "{\varphi}"] \\
    D \ar[r, "{\Psi_{|_{D}}}", "{\cong}"']& Y
    \end{tikzcd}
\end{equation}
we have that the restriction of $\Psi$ on $C_m$ takes a divisor that is linearly equivalent to the ramification divisor of $\phi$ into a divisor that is linearly equivalent to the ramification divisor of $\varphi$. Consider
\begin{equation}\label{Sec2Prop4Rphi}
 R_{\phi} \sim ((m-2)\Gamma_0 + ((m-1) E + q)\mathfrak{f})_{|_{C_m}} \sim ((m-1)(\Gamma_0 + E \mathfrak{f}) + (q \mathfrak{f} - \Gamma_0))_{|_{C_m}} \, .
\end{equation}
Since $\Gamma_0$ and $C_m$ meet exactly at the point $q_0$ in which the fiber $q\mathfrak{f}$ meets $\Gamma_0$, it follows that $(q \mathfrak{f} - \Gamma_0))_{|_{C_m}}$ is an effective divisor on $C_m$ that consists of $m-1$ points, say $q_1, \ldots , q_{m-1}$ on $C_m$, in which $q\mathfrak{f}$ intersects $C_m$ besides $q_0$. Hence,
\[
R_{\phi} \sim ((m-1)(\Gamma_0 + E \mathfrak{f}))_{|_{C_m}} + q_1 + \cdots + q_{m-1} \, .
\]
The morphism $\Psi : S \to \mathbb{P}^r$ is defined by the linear series $|\Gamma_0 + E \mathfrak{f}|$ on $S$, so $\Psi$ maps the restriction $((m-1)(\Gamma_0 + E \mathfrak{f}))_{|_{C_m}}$ to the divisor on $X_m$ cut by a hypersurface of degree $m-1$. Also, $\Psi$ maps the fiber $q \mathfrak{f}$ into the line $l_q$. The images of the points $q_1, \ldots , q_{m-1}$ under $\Psi$ will be the $m-1$ points in which $l_q$ meets $X_m$ besides $P$. Therefore, $R_{\varphi}$ is linearly equivalent to the divisor cut on $X_m$ by a hypersurface of degree $(m-1)$ together with the images of the points $q_1, \ldots , q_{m-1}$, which lie on $l_q$ as claimed.

 \item Since $\deg X_m = (\Gamma_0 + E\mathfrak{f}) \cdot (m\Gamma_0 + (mE+q)\mathfrak{f}) = me+1$, it follows by {\rm (a)} that
 \[
   \begin{aligned}
   \deg R_{\varphi} & = (m-1) \deg X_m + (m-1)
   = (m-1)(me + 2) \, .
  \end{aligned}
 \]

 \item To verify the last claim we show for the branch divisor $\phi (R_{\phi})$ of $\phi : C_m \to D$ that
 \[
  \phi (R_{\phi}) \sim ((m(m-1)E + 2(m-1)q)\mathfrak{f})_{|_{D}} \, .
 \]

\noindent Recall first that the map $\phi : \Div (C_m) \to \Div (D)$ is linear in the sense that $\phi(\sum\limits_{j} n_{P_j} P_j) = \sum\limits_{j} n_{P_j} \phi(P_j)$, where $P_j \in C_m$ and $n_j \in \mathbb{Z}$. Note also that according to \cite[Ex. IV.2.6, p. 306]{Hart77}, the image of any divisor linearly equivalent to $\sum\limits_{j} n_{P_j} P_j$ determines the  linear equivalence class of $\phi(\sum\limits_{j} n_{P_j} P_j)$. Thus, as we claim just linear equivalence, the first equivalence in equation {\rm(\ref{Sec2Prop4Rphi})} implies that it is sufficient to verify that
 \begin{enumerate}[label=(\arabic*), leftmargin=*, font=\rmfamily]
  \item $\phi \left( (m-2)(\Gamma_0 + E\mathfrak{f})_{|_{C_m}} \right) \sim ((m(m-2)E + (m-2)q)\mathfrak{f})_{|_{D}}$, and
  \item $\phi \left( (E+q)\mathfrak{f} )_{|_{C_m}} \right) \sim (( m(E+q))\mathfrak{f})_{|_{D}}$.
 \end{enumerate}
The first claim follows from the fact that $\Gamma_0$ and $C_m$ intersect exactly at $q_0$, $\phi (q_0) = q\mathfrak{f}_{|_D}$ and that $\phi : C_m \to D$ is an $m:1$ covering. The second claim follows by similar reasons. This implies about the branch divisor on $D$ that
\[
 \begin{aligned}
  \phi (R_{\phi}) & \sim \phi ((m-2)(\Gamma_0 + E\mathfrak{f})_{|_{C_m}}) + \phi \left( ((E+q)\mathfrak{f} )_{|_{C_m}} \right) \\
  & \sim ((m(m-2)E + (m-2)q)\mathfrak{f})_{|_{D}} + (( m(E+q))\mathfrak{f})_{|_{D}} \\
  & \sim ((m(m-1)E + 2(m-1)q)\mathfrak{f})_{|_{D}} \, .
 \end{aligned}
\]
By the commutativity of diagram {\rm (\ref{Sec2PropCommutDiag})} we have that
\[
\varphi (R_{\varphi}) \sim \Psi_{|_D} (\phi (R_{\phi})) \sim
\Psi   ((m(m-1)E + 2(m-1)q)\mathfrak{f})_{|_{D}} )  \, .
\]
Recall that $\Gamma_0$ and $D$ do not intersect and
$E\mathfrak{f}_{|_{D}} \sim (\Gamma_0 + E\mathfrak{f})_{|_{D}}$.  Since the divisors from $|\Gamma_0 + E\mathfrak{f}|$ are mapped by $\Psi$ into hyperplane sections of $F$, it follows that the branch divisor $\varphi (R_{\varphi})$ is linearly equivalent to a divisor on $Y$ cut by a hypersurface of degree $m(m-1)$ together with the with the divisor $2(m-1)Q$, where $Q$ is the point in which the line $l_q$ meets $Y$. Therefore, $
  \mathcal{O}_Y (\varphi (R_{\varphi})) \cong \mathcal{O}_Y (m(m-1)) \otimes \mathcal{O}_Y (2(m-1)Q)
$ as it was claimed.

 \item The claim and its proof are contained in the proof of \cite[Proposition 2]{CIK22}.
\end{enumerate}
\end{proof}

The proposition that follows will be used in the proof of the {\rm Main Theorem} to identify the curves of given degree and genus that lie a cone in terms of the linear equivalence class of a specific divisor on the desingularization of the cone.

\begin{prop}\label{Sec2PropIdentifyCurves}
Suppose that $F \subset \mathbb{P}^r$ is a cone over a smooth integral linearly normal curve $Y$ of genus $\gamma$ and degree $e \geq 2\gamma + 1$ in $\mathbb{P}^{r-1}$. Let $S$ be the ruled surface defined as the blow-up of $F$ at its vertex, and let $f : S \to Y$ be the natural surjective morphism with a section $Y_0$ of minimal self-intersection. If $X$ is a smooth integral curve of degree $d = me + 1$ and genus $g = \binom{m}{2}e + m\gamma$ on $F$, then its proper transform $C$ on $S$ is linearly equivalent to $m Y_0 + (mE+q)\mathfrak{f}$, where $E$ is an effective divisor of degree $e$ on $Y$ such that $S \cong \mathbb{P} (\mathcal{O}_Y \oplus \mathcal{O}_Y (-E))$ and $q$ is a point on $Y$.
\end{prop}
\begin{proof}
Since $S$ is the blow-up of $F$ at its vertex, it must be a decomposable ruled surface over $Y$. Since $Y_0$ is the section of minimal self-intersection of $f : S \to Y$, we must have that $\deg E = -Y^2_0 = \deg Y = e$. The Picard group of $S$ is generated by $Y_0$ and the pull-backs via $f^{\ast}$ of the divisors on $Y$. Hence, $C \sim a Y_0 + B \mathfrak{f}$ for a divisor $B$ on $Y$. For the degree of $X$ we have
\[
 me + 1 = \deg X = (Y_0 + E\mathfrak{f}) \cdot (aY_0 + B\mathfrak{f}) = -ae + ae + \deg B \, ,
\]
so $\deg B = me + 1$. Applying the adjunction theorem for $C$ we get
\[
 \begin{aligned}
  2g-2 & = (K_{C} + C) \cdot C \\
  & = (-2Y_0 + (K_{Y} - E) \mathfrak{f} + aY_0 + B \mathfrak{f}) \cdot (a Y_0 + B \mathfrak{f}) \\
  & = ((a-2)Y_0 + (K_{Y} - E + B)\mathfrak{f}) \cdot (a Y_0 + B \mathfrak{f}) \\
  & = a(a-2)(-e) + (me+1)(a-2) + a(2\gamma-2 -e + me+1) \\
  & = -ea^2 + 2ae + (me+1)a - 2me-2 + (2\gamma-1+(m-1)e)a \, .
 \end{aligned}
\]
Since $2g-2 = m(m-1)e + 2m\gamma - 2$, we obtain
\begin{equation}\label{Sec2EqFora}
 ea^2 - ((2m+1)e+2\gamma)a + m(m+1)e + 2m\gamma = 0 \, .
\end{equation}
Solving {\rm (\ref{Sec2EqFora})} for $a$ we obtain solutions $a = m$ and $a = m+1 + \frac{2\gamma}{e}$. Since $e \geq 2\gamma + 1$, the second number is not an integer, so $a = m$ is the only solution.

It remains to prove the claim about $B$, that is, $B \sim mE + q$ for some point $q \in Y$. An argument similar to that in \cite[Prop. V.2.6, p.371]{Hart77} shows that $j_{\ast} \mathcal{O}_{Y_0} (Y_0) \cong \mathcal{O}_{Y} (-E)$, where $j$ is the isomorphism $j : Y_0 \to Y$. Namely, consider the exact sequence
\[
 0 \to \mathcal{O}_{S} \to \mathcal{O}_{S}(Y_0) \to \mathcal{O}_{Y_0} \otimes \mathcal{O}_{S} (Y_0) \to 0 \, ,
\]
and push it down to $Y$. By Grauert's theorem we have
\[
 0 \to f_{\ast} \mathcal{O}_{S} \to f_{\ast} \mathcal{O}_{S}(Y_0) \to j_{\ast} ( \mathcal{O}_{Y_0} (Y_0) ) \to 0 \, .
\]
Since  $f_{\ast} \mathcal{O}_{S} \cong \mathcal{O}_{Y}$
and $f_{\ast} \mathcal{O}_{S}(Y_0) \cong \mathcal{O}_{Y} \oplus \mathcal{O}_{Y} (-E)$, we get $j_{\ast} \mathcal{O}_{Y_0} (Y_0) \cong \mathcal{O}_{Y} (-E)$. Further, $C$ is a smooth curve on $S$ and $C \cdot Y_0 = (m Y_0 + B \mathfrak{f}) \cdot Y_0 = 1$, so $C$ intersects $Y_0$ in a single point, say $z = C \cap Y_0$. Since $C \sim m Y_0 + B \mathfrak{f}$, the restrictions $C_{|_{Y_0}}$ and $(m Y_0 + B \mathfrak{f})_{|_{Y_0}}$ must be linearly equivalent too. Hence,
\[
 z \sim (m Y_0 + B \mathfrak{f})_{|_{Y_0}} \, ,
\]
or equivalently, $j (z) \sim -mE + B$ on $Y_0$. Taking $q := j (z)$, we obtain $B \sim mE + q$.
\end{proof}

In the proof of the main theorem in section \ref{Sec3} we will need the exact form of $\varphi_{\ast} \mathcal{O}_{X_m}$ and $\varphi_{\ast} (\mathcal{O}_{X_m} (P))$ for $m = 3$.  The statement giving the explicit expressions of those bundles develops an idea encountered in \cite[Proposition 2.2]{FP05MN}. Due to obvious reasons, we give a formulation and a proof only in the case $m = 3$, which is sufficient for our purposes.

\begin{prop}\label{Sec2PropPushForwardStrSheaf}
Assume the conditions and notations {\rm (MS)} and {\rm (AN)}. Fix $m = 3$ and denote $C_3 =: C$ and $X_3  =: X$. Then
\begin{enumerate}[label=\normalfont(\alph*), leftmargin=*, font=\rmfamily]
 \item $\varphi_{\ast} (\mathcal{O}_{X} (P)) \cong \mathcal{O}_{Y} \oplus \mathcal{O}_{Y} (-1) \oplus ( \mathcal{O}_{Y} (-2) \otimes \mathcal{O}_{Y} (-Q) )$,

 \item $\varphi_{\ast} \mathcal{O}_{X} \cong \mathcal{O}_{Y} \oplus (\mathcal{O}_{Y} (-1) \otimes \mathcal{O}_{Y} (-Q) )\oplus ( \mathcal{O}_{Y} (-2) \otimes \mathcal{O}_{Y} (-Q) )$.
\end{enumerate}
\end{prop}
\begin{proof}
The equivalent statements about $\phi : C \to D$ appear as
\begin{enumerate}[label=(\alph*'), leftmargin=*, font=\rmfamily]
 \item $\phi_{\ast} (\mathcal{O}_{C} (\Gamma_0)) \cong \mathcal{O}_{D} \oplus \mathcal{O}_{D} (-E\mathfrak{f}) \oplus \mathcal{O}_{D} (-(2E+q)\mathfrak{f} )$,

 \item $\phi_{\ast} \mathcal{O}_{C} \cong \mathcal{O}_{D} \oplus \mathcal{O}_{D} (-(E+q)\mathfrak{f}) \oplus \mathcal{O}_{D} (-(2E+q)\mathfrak{f} )$.
\end{enumerate}
If we denote by $\nu$ the morphism $f_{|_C} : C \to \Gamma$, or equivalently, if $\iota$ is the embedding $\iota : C \hookrightarrow S$ and $\nu$ is the composition $f \circ \iota$, the two claims translate into
\begin{enumerate}[label=(\alph*''), leftmargin=*, font=\rmfamily]
 \item $\nu_{\ast} (\mathcal{O}_{C} (q_0)) \cong \mathcal{O}_{\Gamma} \oplus \mathcal{O}_{\Gamma} (-E ) \oplus \mathcal{O}_{\Gamma} (-2E-q)$,

 \item $\nu_{\ast} \mathcal{O}_{C} \cong \mathcal{O}_{\Gamma} \oplus \mathcal{O}_{\Gamma} (-E-q) \oplus \mathcal{O}_{\Gamma} (-2E-q)$.
\end{enumerate}
It is sufficient to prove claims {\rm (a'')} and {\rm (b'')}, which we will do next. We remark that claim {\rm (b'')} has been proven by Miranda for varieties of arbitrary dimension, see \cite[Proposition 8.1, p.1150]{Mir85}. Here we give a proof of it (for curves) as well, as it is easy to do in our context.

Since $C \in |3\Gamma_0 + (3E + q)\mathfrak{f}|$, there is an exact sequence
\[
 0 \to \mathcal{O}_{S} (-2\Gamma_0 -(3E +q)\mathfrak{f})
   \to \mathcal{O}_{S} ( \Gamma_0)
   \to \iota_{\ast} \mathcal{O}_{C} ( \Gamma_0) \equiv \iota_{\ast} \mathcal{O}_{C} ( q_0 )
   \to 0 \, .
\]
Pushing it down to $\Gamma$ via $f_{\ast}$, we get the exact sequence

{\footnotesize
\begin{equation}\label{PushForward}
\begin{tikzcd}
0 \arrow{r} &
f_{\ast} \mathcal{O}_{S} (-2\Gamma_0 -(3E +q)\mathfrak{f}) \arrow{r} &
f_{\ast} \mathcal{O}_{S} ( \Gamma_0) \arrow{r} \arrow[phantom, ""{coordinate, name=Z}]{d} &
f_{\ast}  \iota_{\ast} \mathcal{O}_{C} ( q_0 )
  \arrow[
    rounded corners,
    to path={
      -- ([xshift=2ex]\tikztostart.east)
      |- (Z) [near end]\tikztonodes
      -| ([xshift=-2ex]\tikztotarget.west)
      -- (\tikztotarget)
    }
  ]{dll} \\
& R^1 f_{\ast} \mathcal{O}_{S} (-2\Gamma_0 -(3E +q)\mathfrak{f}) \arrow{r} &
R^1 f_{\ast} \mathcal{O}_{S} ( \Gamma_0) \arrow{r} &
R^1 f_{\ast} \iota_{\ast} \mathcal{O}_{C} ( q_0)  &
\end{tikzcd}
\end{equation}
}

\bigskip

\noindent For every point $z \in \Gamma$ we have that $z\mathfrak{f} \cong \mathbb{P}^1$ and $\deg (-2\Gamma_0 -(3E +q)\mathfrak{f})_{|_{z\mathfrak{f}}}) = \deg (-2\Gamma_0 \cdot \mathfrak{f})= -2$, hence
\[
 h^i (z\mathfrak{f}, \mathcal{O}_{S} (-2\Gamma_0 -(3E +q)\mathfrak{f})_{|_{z\mathfrak{f}}}) = h^i (\mathbb{P}^1, \mathcal{O}_{\mathbb{P}^1} (-2)) =
 \begin{cases}
  0 & \mbox{ if } i = 0 \\
  1 & \mbox{ if } i = 1\, .
 \end{cases}
\]
By Grauert's theorem, see \cite[Theorem III.12.9]{Hart77}, it follows that the push-forward $f_{\ast} \mathcal{O}_{S} (-2\Gamma_0 -(3E +q)\mathfrak{f})$ vanishes, while $R^1 f_{\ast} \mathcal{O}_{S} (-2\Gamma_0 -(3E +q)\mathfrak{f})$ must be a locally free sheaf of rank one, that is, a line bundle on $\Gamma$. From the definition of $S$ we have $f_{\ast} \mathcal{O}_{S} ( \Gamma_0 ) \cong \mathcal{O}_{\Gamma} \oplus \mathcal{O}_{\Gamma} (-E)$, and since $h^1 (z\mathfrak{f}, \mathcal{O}_{S} (\Gamma_0)_{|_{z\mathfrak{f}}}) = h^1 (\mathbb{P}^1, \mathcal{O}_{\mathbb{P}^1} (1)) = 0$, the Grauert's theorem implies that {\rm (\ref{PushForward})} reduces to
\begin{equation}\label{PushForwardRed}
 0 \to \mathcal{O}_{\Gamma} \oplus \mathcal{O}_{\Gamma} (-E) \to \nu_{\ast} \mathcal{O}_{C} ( q_0 ) \to R^1 f_{\ast} \mathcal{O}_{S} (-2\Gamma_0 -(3E +q)\mathfrak{f}) \to 0 \, .
\end{equation}
Since $\nu : C \to \Gamma$ is a triple covering morphism, the push-forward $\nu_{\ast} \mathcal{O}_{C}$ must split as
\[
 \nu_{\ast} \mathcal{O}_{C} = \mathcal{O}_{\Gamma} \oplus \mathcal{E}^{\vee}
\]
where $\mathcal{E}$ is a vector bundle of rank two on $\Gamma$ for which its dual bundle $\mathcal{E}^{\vee}$ is the Tschirnhausen module of $\nu$. Denote $\beta := \det (\nu_{\ast} \mathcal{O}_{C}) = \det \mathcal{E}^{\vee}$. Using \cite[Ex. IV.2.6(d), p.306]{Hart77}, we obtain easily that $\deg \beta = -3e -2$. Since ${\Gamma_0}$ and $C$ meet exactly at the point $q_0$, which is mapped by $\nu$ into $q$ on $\Gamma$, it follows by \cite[Ex. IV.2.6(a), p.306]{Hart77} that
\[
 \det (\nu_{\ast} \mathcal{O}_{C} ( q_0)) \cong \det (\nu_{\ast} \mathcal{O}_{C}) \otimes \mathcal{O}_{\Gamma} (q) \cong \beta (q) \, .
\]
Therefore,  by \eqref{PushForwardRed},
\[
 \begin{aligned}
  R^1 f_{\ast} \mathcal{O}_{S} (-2\Gamma_0 -(3E +q)\mathfrak{f})
  & \cong \det (\nu_{\ast} \mathcal{O}_{C} ( q_0)) \otimes (\det (\mathcal{O}_{\Gamma} \oplus \mathcal{O}_{\Gamma} (-E))^{-1} \\
  & \cong \beta (q) \otimes \mathcal{O}_{\Gamma} (E) \\
  & \cong \beta (E+q) \, .
 \end{aligned}
\]
Since $\deg (\beta^{\vee} (-2E-q)) = e+1 > 2\gamma -2$, we have
{\small
\[
 \begin{aligned}
  \Ext^1 (R^1 f_{\ast} \mathcal{O}_{S} (-2\Gamma_0 -(3E +q)\mathfrak{f}), & \mathcal{O}_{\Gamma} \oplus \mathcal{O}_{\Gamma} (-E))
  = \Ext^1 (\beta (E+q), \mathcal{O}_{\Gamma} \oplus \mathcal{O}_{\Gamma} (-E)) \\
  & = H^1 (\Gamma , \beta^{\vee} (-E-q) \oplus \beta^{\vee} (-2E-q)) \\
  & = 0 \, .
 \end{aligned}
\]
}
This implies that the exact sequence {\rm (\ref{PushForwardRed})} splits, so we get
\begin{equation}\label{Sec2PropPushForwardSplitSeq}
 \nu_{\ast} \mathcal{O}_{C} (q_0) \cong \mathcal{O}_{\Gamma} \oplus \mathcal{O}_{\Gamma} (-E) \oplus \beta (E+q) \, .
\end{equation}

\noindent Since the Tschirnhausen module $\mathcal{E}^{\vee}$ is determined uniquely by the covering morphism $\nu : C \to \Gamma$ and since $\iota : C \hookrightarrow S = \mathbb{P} (\mathcal{O}_{\Gamma} \oplus \mathcal{O}_{\Gamma}(-E))$ is an embedding, it follows by \cite[Theorem 1.3, p.439]{CE96}  that
\[
    \nu_{\ast} \mathcal{O}_{C} \cong \mathcal{O}_{\Gamma} \oplus ((\mathcal{O}_{\Gamma} \oplus \mathcal{O}_{\Gamma}(-E)) \otimes \mathcal{L} )
\]
for some line bundle $\mathcal{L}$ on $\Gamma$. Using $\deg ( \det \nu_{\ast} \mathcal{O}_{C} ) = \deg \beta = -3e - 2$, we get $\deg \mathcal{L} = -e-1$. From
\[
 0 \to \nu_{\ast} \mathcal{O}_{C} \to \nu_{\ast} \mathcal{O}_{C} (q_0) \to \mathcal{O}_q \to 0 \, ,
\]
we obtain
\[
 0 \to \mathcal{O}_{\Gamma} \oplus \mathcal{L} \oplus \mathcal{L} (-E) \xrightarrow{\tau} \mathcal{O}_{\Gamma} \oplus \mathcal{O}_{\Gamma} (-E) \oplus \beta (E+q) \to \mathcal{O}_q\to 0 \, .
\]
Let $\iota_k$, for $k = 1, 2, 3$, denote the inclusion maps of the summands $\mathcal{O}_{\Gamma}$, $\mathcal{L}$ and $\mathcal{L} (-E)$ into the direct sum $\mathcal{O}_{\Gamma} \oplus \mathcal{L} \oplus \mathcal{L} (-E)$, respectively. Consider the projection $p_3 : \mathcal{O}_{\Gamma} \oplus \mathcal{O}_{\Gamma} (-E) \oplus \beta (E+q) \to \beta (E+q)$. Now, let's examine the compositions $p_3 \circ \tau \circ \iota_k$ for $k = 1, 2, 3$. Since the degree of $\mathcal{O}_{\Gamma}$ is zero, which is greater than $-2e-1$, the degree of $\beta (E+q)$, it follows that $p_3 \circ \tau \circ \iota_1$ is the zero map. Similarly, with $\deg \mathcal{L} = -e-1$, which is greater than $-2e-1$, the degree of $\beta(E+q)$, the map $p_3 \circ \tau \circ \iota_2$ is also the zero map. Therefore, it remains that the map $p_3 \circ \tau \circ \iota_3 : \mathcal{L} (-E) \to \beta (E+q)$ is nonzero. Given that $\deg \mathcal{L}(-E) = -2e-1 = \deg \beta(E+q)$, we conclude that $\mathcal{L} (-E) \cong \beta (E+q)$. Additionally, we have $\beta = \det (\nu_{\ast} \mathcal{O}_C) = \mathcal{L}^2 (-E)$, implying $\mathcal{L} \cong \mathcal{O}_{\Gamma} (-E-q)$. This implies $\beta \cong \mathcal{O}_{\Gamma} (-3E-2q)$. Hence, both statements (a) and (b) in the proposition are established.
\end{proof}

\medskip

Finally, we recall one more result that will be used in the proof of the {\rm Main Theorem} in section \ref{Sec3}.
\begin{prop}\label{PropInnerProjSESNormBund}(\cite[Proposition 2]{CIK22})
Let $X$ be a non-degenerate smooth integral curve in $\mathbb{P}^r$, where $r \geq 3$. Let $H$  be a hyperplane in $\mathbb{P}^r$ and $P$ be a point on $X$. Suppose that the inner projection $\varphi \, : \, X \to H \cong \mathbb{P}^{r-1}$ with center $P$ maps $X$ to a non-degenerate smooth integral curve $Y$ in $H$. Denote by $R_{\varphi}$ the ramification divisor of $\varphi$. Then
\[
    0 \to \mathcal{O}_X (R_{\varphi}) \otimes \mathcal{O}_X (1) \otimes \mathcal{O}_X (2P)
      \to N_{X / \mathbb{P}^r}
      \to \varphi^{\ast} N_{Y / \mathbb{P}^{r-1}} \otimes \mathcal{O}_X (P)
      \to 0 \, ,
\]
where $N_{X / \mathbb{P}^r}$ is the normal bundle of $X$ in $\mathbb{P}^r$ and $N_{Y / \mathbb{P}^{r-1}}$ is the normal bundle of $Y$ in $H \cong \mathbb{P}^{r-1}$.
\end{prop}

\bigskip

\section{ Proof of the main theorem }\label{Sec3}

Recall the basic numerical assumptions in the theorem: $\gamma \geq 3$ and $e \geq 4\gamma + 5$. Throughout this section we also fix
\[
  g := 3e + 3 \gamma \, , \quad d := 3e+1 = g - 3\gamma + 1 \quad \mbox{ and } \quad r := e-\gamma + 1 = \frac{g}{3} - 2\gamma + 1 \, .
\]
The technique used in the proof is derived from  \cite{CIK22}, \cite{CLM96} and \cite{Cil87}.
The proof itself proceeds in three main steps:
\begin{itemize}[font=\sffamily, leftmargin=1.8cm, style=nextline]
 \item[{\rm \bf Step I.}] We construct a family $\mathcal{F}$ of curves satisfying the characterization {\rm (iii)} in the main theorem, then we consider the closure $\mathcal{H}$ of the subset of $\mathcal{I}_{d, g, r}$ parametrizing
 the family $\mathcal F$ and show that
 \[
    \dim \mathcal{H} =  r^2 + 7e + 4 \, .
 \]

 \item[{\rm \bf Step II.}] For a general curve $X$ from the family $\mathcal{F}$ we show that
 \[
  \dim T_{[X]} \mathcal{H} = h^0 (X, N_{X / \mathbb{P}^{r}})
    = r^2 + 7e + 5 = \dim \mathcal{H} + 1 \, .
 \]

 \item[{\rm \bf Step III.}] We show that $\mathcal{H}$ forms an irreducible component of $\mathcal{I}_{d, g, r}$.
\end{itemize}

\bigskip

{\bf Step I.} Construction of the family.

Let $\Gamma \in \mathcal{M}_{\gamma}$ be a general curve of genus $\gamma$ and $E$ be a general divisor of degree $e \geq 4\gamma + 5$ on $\Gamma$. Let $q \in \Gamma$. Consider the ruled surface $S := \mathbb{P} (\mathcal{O}_{\Gamma} \oplus \mathcal{O}_{\Gamma}(-E))$ with natural projection $f : S \to \Gamma$. Denote by $\Gamma_0$ the section of minimal self-intersection on $S$, that is, $\Gamma^2_0 = -e$. As it was mentioned in section \ref{Sec2}, $\Pic (S) \cong \mathbb{Z}[\Gamma_0] \oplus {f}^{\ast} (\Pic (\Gamma))$. Just as there, for a divisor $\Delta \in \Div (\Gamma)$ we denote by $\Delta \mathfrak{f}$ the divisor ${f}^{\ast} (\Delta)$ on $S$. Consider the morphism $\Psi := \Psi_{|\Gamma_0 + E\mathfrak{f}|} : S \to \mathbb{P}^r$ determined by the linear series $\mathcal{O}_S (\Gamma_0 + E\mathfrak{f})$ on $S$. Define $\mathcal{F}$ as the family of curves that are images of the divisors from the linear series $|3{\Gamma}_0 + (3E + q){\mathfrak{f}}|$ on $S$ under the morphism $\Psi$, by varying $\Gamma$ in $\mathcal{M}_{\gamma}$, running $E$ through the set of general effective divisors of degree $e$ on $\Gamma$ and $q \in \Gamma$.
Note that  a general $X \in \mathcal{F}$ satisfies the properties ${\rm (iii)}$  in the theorem. More precisely, we get two properties
 ${\rm (iii)}. 1$ and ${\rm (iii)}.2$  from  Propositions \ref{Sec2PropCatGim} and  \ref{SecPropLine}, the property ${\rm (iii)}.3$ from the discussion just before Proposition  \ref{Sec2PropRamBranch},  and the property ${\rm (iii)}.4$  from   Proposition  \ref{Sec2PropRamBranch}, (a).

We now compute the dimension of $ \mathcal{F}$ in what follows:

\noindent $\dim \mathcal{F} = $
\begin{itemize}[font=\sffamily, leftmargin=1.3cm, style=nextline]
     \item[$ + $] $3\gamma - 3$ \ : \ number of parameters of curves $\Gamma \in \mathcal{M}_{ \gamma }$

     \item[$ + $] $\gamma$ \ : \ number of parameters of line bundles $\mathcal{O}_{\Gamma} (E) \in \Pic (\Gamma)$ of degree $e \geq 4\gamma + 5$
     necessary to fix the geometrically ruled surface $\mathbb{P} (\mathcal{O}_{\Gamma} \oplus \mathcal{O}_{\Gamma} (-E))$

     \item[$ + $] $(r+1)^2 - 1 = \dim (\Aut (\mathbb{P}^{r}))$

     \item[$ + $] $1$ \ : \ number of parameters necessary to fix $q \in \Gamma$

     \item[$ - $] $(e- \gamma + 2) = \dim G_{F}$, where $G_{F}$ is the subgroup of $\Aut (\mathbb{P}^{r})$ fixing the scroll $F$, see \cite[Lemma 6.4, p. 148]{CCFM2009}

     \item[$ + $] $6e - 3\gamma + 6 = \dim |3{\Gamma}_0 + (3E + q){\mathfrak{f}}|$ \ : \ number of parameters to choose a curve in the linear equivalence class of $3{\Gamma}_0 + (3E + q){\mathfrak{f}}$ on $S$.
\end{itemize}

\noindent Define $\mathcal{H}$ as the closure in $\mathcal{I}_{d, g, r}$ of the set parametrizing $\mathcal{F}$. Accounting the above numbers we get
\[
 \dim \mathcal{H} = \dim \mathcal{F}
 = r^2 + 7e + 4 \, ,
\]
whence {\rm Step I} is completed.

\medskip

{\bf Step II.} Computation of the tangent space to $\mathcal{H}$.

Let $X \in \mathcal{F}$ be a general curve from the family, that is, $X$ is the image $\Psi (C)$ of a general $C \in |3\Gamma_0 + (3E + q)\mathfrak{f}|$ on $S$, the base curve $\Gamma \in \mathcal{M}_{\gamma}$ is general, and $E \in \Div^e (\Gamma)$ and $q \in \Gamma$ are also general. Also, $X$ lies on the cone $F := \Psi (S)$ over a curve $Y \subset \mathbb{P}^{r-1}$ that is the image $Y := \Psi (D)$ of a general $D \in |\Gamma_0 + E\mathfrak{f}|$. Let $l_q$ be the line from the ruling of $F$ that is the image of $q\mathfrak{f}$ and $Q = l_q \cap Y$. Denote by $\varphi : X \to Y$ the projection with center $P$ of $X$ to the hyperplane containing $Y$. It is a $3:1$ covering morphism. Recall that by {\rm Proposition \ref{Sec2PropRamBranch}} its ramification divisor $R_{\varphi}$ is linearly equivalent to the divisor on $X$ cut by a quadric hypersurface and the two points, say $Q_1$ and $Q_2$, besides $P$, in which the line $l_q$ meets $X$. Applying {\rm Proposition \ref{PropInnerProjSESNormBund}}, we obtain the short exact sequence
\begin{equation}\label{Sec3InnerProjSESNormBund}
    0 \to \mathcal{O}_X (3) \otimes \mathcal{O}_X (Q_1 + Q_2 + 2P)
      \to N_{X / \mathbb{P}^r}
      \to \varphi^{\ast} N_{Y / \mathbb{P}^{r-1}} \otimes \mathcal{O}_X (P)
      \to 0 \, ,
\end{equation}
in which $N_{X / \mathbb{P}^r}$ is the normal bundle of $X$ in $\mathbb{P}^r$ and  $N_{Y / \mathbb{P}^{r-1}}$ is the normal bundle of $Y$ in $\mathbb{P}^{r-1}$. Due to $e \geq 2\gamma + 1$, the Hilbert scheme $\mathcal{I}_{e, \gamma, e-\gamma}$ is irreducible and generically smooth of the expected dimension $\dim \mathcal{I}_{e, \gamma, e-\gamma} = \lambda_{e, \gamma, r-1} = er - (r-4)(\gamma - 1)$. Since $\Gamma \in \mathcal{M}_{\gamma}$ is general and $Y$ is isomorphic to $\Gamma$ and $\deg Y = e$, it follows that
\[
    h^0 (Y, N_{Y / \mathbb{P}^{r-1}}) = er - (r-4)(\gamma - 1) \, .
\]

\noindent For the degree of the line bundle $\mathcal{O}_X (3) \otimes \mathcal{O}_X (Q_1 + Q_2 + 2P)$ in {\rm (\ref{Sec3InnerProjSESNormBund})} we have
\[
 \deg \left( \mathcal{O}_X (3) \otimes \mathcal{O}_X (Q_1 + Q_2 + 2P) \right)
 = 3 \deg X + 4 = 3(3e+1) + 4 = 9e + 7 \, .
\]
According to the assumptions in the theorem, $g = 3e + 3\gamma$ and $e \geq 4\gamma + 5$, so for the degree of the line bundle $\mathcal{O}_X (3) \otimes \mathcal{O}_X (Q_1 + Q_2 + 2P)$ we obtain $9e + 7 > 2g-2$. Thus, $\mathcal{O}_X (3) \otimes \mathcal{O}_X (Q_1 + Q_2 + 2P)$ is nonspecial, hence
{\small
\begin{equation}\label{h0NXinPr}
 h^0 (X, N_{X / \mathbb{P}^r}) = h^0 (X, \mathcal{O}_X (3) \otimes \mathcal{O}_X (Q_1 + Q_2 + 2P)) + h^0 (X, \varphi^{\ast} N_{Y / \mathbb{P}^{r-1}} \otimes \mathcal{O}_X (P)) \, .
\end{equation}
}
By the Riemann-Roch theorem
\[
 h^0 (X, \mathcal{O}_X (3) \otimes \mathcal{O}_X (Q_1 + Q_2 + 2P))
 = 6e - 3\gamma + 8 \, .
\]
To compute $h^0 (X, \varphi^{\ast} N_{Y / \mathbb{P}^{r-1}} \otimes \mathcal{O}_X (P) )$ we use the projection formula, that is,
\[
\begin{aligned}
 h^0 (X, \varphi^{\ast} N_{Y / \mathbb{P}^{r-1}} \otimes \mathcal{O}_X (P) )
 & = h^0 (Y, {\varphi}_{\ast} ( \, {{\varphi}}^{\ast} N_{Y / \mathbb{P}^{r-1}} \otimes \mathcal{O}_{X} (P) \, ) ) \\
 & = h^0 (Y, N_{Y / \mathbb{P}^{r-1}} \otimes {\varphi}_{\ast} \mathcal{O}_{X} (P) ) \, .
\end{aligned}
\]
By {\rm Proposition \ref{Sec2PropPushForwardStrSheaf}} we have $\varphi_{\ast} (\mathcal{O}_{X} (P)) \cong \mathcal{O}_{Y} \oplus \mathcal{O}_{Y} (-1) \oplus ( \mathcal{O}_{Y} (-2) \otimes \mathcal{O}_{Y} (-Q) )$, so it follows that
\begin{multline*}
 h^0 (Y, N_{Y / \mathbb{P}^{r-1}} \otimes {\varphi}_{\ast} \mathcal{O}_{X} (P) ) \\
 = h^0 (Y, N_{Y / \mathbb{P}^{r-1}}) + h^0 (Y, N_{Y / \mathbb{P}^{r-1}}(-1)) + h^0 (Y, N_{Y / \mathbb{P}^{r-1}}(-2) \otimes \mathcal{O}_{Y} (-Q) ) \, .
\end{multline*}
Since $Y \cong \Gamma$ is general in $\mathcal{M}_{\gamma}$, $\gamma \geq 3$, and $E$ is a general divisor of degree $e \geq 4\gamma+5$ on $\Gamma$, it follows by \cite[Proposition 2.1 and Proposition 2.12]{CLM96} that $h^0 (Y, N_{Y / \mathbb{P}^{r-1}} (-1)) = r$ and $h^0 (Y, N_{Y / \mathbb{P}^{r-1}} (-2)) = 0$. The last implies $h^0 (Y, N_{Y / \mathbb{P}^{r-1}} (-2) \otimes \mathcal{O}_{Y} (-Q) ) = 0$. Using that $r - 1 = e-\gamma$, we find
\[
\begin{aligned}
 h^0 (X, {{\varphi}}^{\ast} N_{Y / \mathbb{P}^{r-1}} \otimes \mathcal{O}_{X} (P))
 & = h^0 (Y, N_{Y / \mathbb{P}^{r-1}}) + h^0 (Y, N_{Y / \mathbb{P}^{r-1}}(-1)) \\
 & = er - (r-4)(\gamma - 1) + r \\
 & = r^2 + r + 4\gamma - 4 \, .
\end{aligned}
\]
The exact sequence {\rm (\ref{h0NXinPr})} then gives $h^0 (X, N_{X / \mathbb{P}^r}) = (6e - 3\gamma + 8) + (r^2 + r + 4\gamma - 4) = r^2 + r + \gamma + 6e + 4 = r^2 + (e - \gamma + 1) + \gamma + 6e + 4 = r^2 + 7e + 5$. Therefore,
\begin{equation}\label{Sec3DimTH}
 \dim T_{[X]} \mathcal{H} = \dim \mathcal{H}+1 = r^2 + 7e + 5 \, .
\end{equation}
This completes {\rm Step II}.

\medskip

{\bf Step III.} Showing that $\mathcal{H}$ forms an irreducible component of $\mathcal{I}_{d, g, r}$.

By definition, $\mathcal{H} \subset \mathcal{I}_{d,g,r}$ is the closure of the set parametrizing smooth integral curves of degree $d$ and genus $g$ on cones in $\mathbb{P}^{r}$ over the curves parametrized by $\mathcal{I}_{e, \gamma, r-1}$, as a general $[X] \in \mathcal{H}$ is in the linear equivalence class $3\Gamma_0 + (3E+q)\mathfrak{f}$ on the desingularization $S$ of a cone $F \subset \mathbb{P}^r$ over $Y \subset \mathbb{P}^{r-1}$ for a general $[Y] \in \mathcal{I}_{e, \gamma, r-1}$. The set $\mathcal{H}$ is clearly irreducible. To show that it is a component, we use that that every flat deformation of a curve from $\mathcal{F}$ is a again a curve on a cone in $\mathbb{P}^r$ over a curve from $\mathcal{I}_{e, \gamma, r-1}$.

\begin{lemma}\label{Sec3DeformLemma}
Let $p_{\mathcal{X}} : \mathcal{X} \to T$ be a flat family of projective curves in $\mathbb{P}^r$ for which there exists a closed point $t_0 \in T$ such that:
\begin{enumerate}[label=(\roman*), leftmargin=*, font=\rmfamily]
 \item  $\mathcal{X}_{t_0}$ is a smooth integral projectively normal curve of genus $g = 3e + 3\gamma$ and degree $3e + 1$;

 \item $\mathcal{X}_{t_0}$ is contained in a cone $F$ over a curve $Y$ corresponding to a general point of $\mathcal{I}_{e, \gamma, r-1}$.
\end{enumerate}
Then there is a neighborhood $U$ of $t_0$ in $T$ such that, for all closed points $t \in U$, $\mathcal{X}_t$ is again a curve on a cone over a smooth integral projectively normal curve of genus $\gamma$ and degree $e$ in $\mathbb{P}^{r-1}$.
\end{lemma}

Assuming the validity of {\rm Lemma \ref{Sec3DeformLemma}}, the proof of {\rm Step III} proceeds as follows. Suppose that
 $\tilde{X} \in U$
is a flat deformation of $X$. {\rm Lemma \ref{Sec3DeformLemma}} implies that $\tilde{X}$ is contained in a cone $\tilde{F} \subset \mathbb{P}^r$ over a curve $\tilde{Y}$, where $[\tilde{Y}] \in \mathcal{I}_{e, \gamma, r-1}$ is general. Let $\tilde{S}$ be the desingularization of $\tilde{F}$ and $\tilde{C}$ be the proper transform of $\tilde{X}$ on $\tilde{S}$. {\rm Proposition \ref{Sec2PropIdentifyCurves}} implies that $\tilde{C} \sim 3\tilde{Y}_0 + (3 \tilde{E} + \tilde{q})\tilde{\mathfrak{f}}$, where $\tilde{\mathfrak{f}} : \tilde{S} \to \tilde{Y}$ is the corresponding surjective morphism, $\tilde{Y}_0$ is the section of minimal self-intersection, $\tilde{Y}^2_0 = -e$, $\tilde{E}$ is a divisor of $\tilde{Y}$ of degree $e$ such that $\tilde{S} \cong \mathbb{P} (\mathcal{O}_{\tilde{Y}} \oplus \mathcal{O}_{\tilde{Y}} (-\tilde{E}))$ and $\tilde{q}$ is a point on $\tilde{Y}$. Also, $\tilde{X}$ is the image of a curve in the linear series $|3\tilde{Y}_0 + (3 \tilde{E} + \tilde{q})\tilde{\mathfrak{f}}|$ under the morphism associated to $|\tilde{Y}_0 + \tilde{E} \tilde{\mathfrak{f}}|$. Because of the definition of $\mathcal{F}$, the above means that $\tilde{X}$ is a curve from the same family. Therefore $\mathcal{H}$ is a component of $\mathcal{I}_{3e+1, g, e-\gamma+1}$.

To complete {\rm Step III} in the proof of the theorem it remains to prove the lemma.

\begin{proof}[Proof of Lemma \ref{Sec3DeformLemma}]
To a large extent our proof repeats the steps of the proofs of similar statements given in \cite[Proposition 1.6, p.354--356]{Cil87} and \cite[Proposition 4.1, p.176--178]{CLM96}. For this reason, we refer, whenever possible, to the statements formulated and proved there.
The statement is local, so we can assume that $T = \Spec (A)$ for a Noetherian ring $A$. Thus, we have a flat family
\[
 \mathcal{X} \subset \Proj A[x_0, x_1, \ldots , x_r] =: \mathbb{P}^r_A \, .
\]
Since projective normality is an ``open property'' and $\mathcal{X}_{t_0}$ is supposed to be projectively normal, we can assume further that the family $\mathcal{X}$ consists of projectively normal curves. By \cite[Ex. III.9.5, p.267]{Hart77} the family $\mathcal{X}$ must be \emph{very flat}. In particular, the number of generators in any degree $n$ of the ideal $I(\mathcal{X}_t)$ of a curve $\mathcal{X}_t \subset \mathbb{P}^r$ from the family is the same for all $t \in T$. Consider the homogeneous ideal $I(\mathcal{X})$ of $\mathcal{X}$ in the ring
\[
    R := A[x_0, x_1, \ldots , x_r]
\]
and let $I(\mathcal{X})_2$ be the vector space of its elements of degree two, that is,
\[
    I(\mathcal{X})_2 := H^0 (\mathbb{P}^r_A, \mathcal{I}_{\mathcal{X}} (2)) \, ,
\]
where $\mathcal{I}_{\mathcal{X}}$ is the ideal sheaf of $\mathcal{X}$. Take $J \subset R$ to be the ideal
\[
 J := \langle I(\mathcal{X})_2 \rangle \,
\]
generated by the elements of degree two. Consider the closed subscheme $\mathcal{W} \subset \mathbb{P}^r_A$ defined as
$
\mathcal{W} := \Proj (R / J) \subset \mathbb{P}^r_A \, .
$
It is indeed a family $p_\mathcal{W} : \mathcal{W} \to T$ parametrized by $T = \Spec (A)$ and we have a commutative diagram
\[
    \begin{tikzcd}[cells={anchor=west}]
      \mathcal{X} \subset \mathcal{W} \subset \mathbb{P}^r_A
   \ar[d, "{p_\mathcal{X}}"', start anchor={[shift={(12pt,-6pt)}]west}, end anchor={[shift={(18pt,-4pt)}]north west}]
   \ar[d, "{p_\mathcal{W}}", start anchor={[shift={(37pt,-6pt)}]west}, end anchor={[shift={(-28pt,-4pt)}]north east}] \\
    \Spec (A)
    \end{tikzcd}
\]
The goal is to show that $p_\mathcal{W} : \mathcal{W} \to T = \Spec (A)$ is a flat family.

By assumption $\mathcal{X}_{t_0}$ is a smooth curve of genus $g = 3e + 3\gamma$ and degree $3e + 1$ contained in a cone $F$ over a smooth integral projectively normal curve $Y$ of genus $\gamma$ and degree $e$ in $\mathbb{P}^{r-1}$. By {\rm Proposition \ref{Sec2PropIdentifyCurves}} this means that the proper transform of $\mathcal{X}_{t_0}$ on the desingularization $S$ of $F$ is in the linearly equivalence class of $3Y_0 + (3E + q)\mathfrak{f}$, where, just as before, $f : S \to Y$ is the surjective morphism for the decomposable ruled surface $S$, $Y_0$ is the section of minimal self-intersection, $E$ is a divisor of degree $e$ on $Y$ such that $S \cong \mathbb{P}(\mathcal{O}_Y \oplus \mathcal{O}_Y (-E))$ and $q \in Y$ is a point. Since $Y$ is a general curve of genus $\gamma \geq 3$ and degree $e \geq 4\gamma + 5$ in $\mathbb{P}^{r-1}$, it follows by \cite{GL86} that the first several terms of the minimal free resolution of its ideal sheaf $\mathcal{I}_{Y}$ appear as
\[
 \cdots \to \bigoplus\limits^{\beta_3}_{j=1} \mathcal{O}_{\mathbb{P}^{r-1}} (-4) \to \bigoplus\limits^{\beta_2}_{j=1} \mathcal{O}_{\mathbb{P}^{r-1}} (-3) \to \bigoplus\limits^{\beta_1}_{j=1} \mathcal{O}_{\mathbb{P}^{r-1}} (-2) \to \mathcal{I}_{Y} \to 0
\]
where $\beta_1, \beta_2, \ldots $ are the \emph{Betti numbers}. By \cite[Proposition 2, p. 232]{CG99} it follows that the first several terms of the minimal free resolution of the ideal sheaf $\mathcal{I}_{\mathcal{X}_{t_0}}$ of $\mathcal{X}_{t_0} \subset \mathbb{P}^r$ are
\begin{equation}\label{Sec3ResolIdSheafX}
  \begin{tikzcd}[row sep=0pt, cells={anchor=west}]
  \cdots \ar[r] & \mathcal{P}_3 \ar[r] & \mathcal{P}_2 \ar[r] & \mathcal{P}_1 \ar[r] & \mathcal{I}_{\mathcal{X}_{t_0}} \ar[r] & 0 \, ,
\end{tikzcd}
\end{equation}
where
\begin{itemize}
 \item $\mathcal{P}_1 = \bigoplus\limits^{r-1}_1 \mathcal{O}_{\mathbb{P}^{r}}(-4) \oplus \bigoplus\limits^{\beta_1}_{j=1} \mathcal{O}_{\mathbb{P}^{r}} (-2)$

 \item $\mathcal{P}_2 = \bigoplus\limits^{\binom{r-1}{2}}_1 \mathcal{O}_{\mathbb{P}^{r}}(-5) \oplus \bigoplus\limits^{\beta_2}_{j=1} \mathcal{O}_{\mathbb{P}^{r}} (-3) \oplus \bigoplus\limits^{\beta_1}_1 \mathcal{O}_{\mathbb{P}^{r}}(-5)$

 \item $\mathcal{P}_3 = \bigoplus\limits^{\binom{r-1}{3}}_1 \mathcal{O}_{\mathbb{P}^{r}}(-6) \oplus \bigoplus\limits^{\beta_3}_{j=1} \mathcal{O}_{\mathbb{P}^{r}} (-4) \oplus \bigoplus\limits^{\beta_2}_1 \mathcal{O}_{\mathbb{P}^{r}}(-6)$
\end{itemize}
To deduce the flatness of the family $p_{\mathcal{W}} : \mathcal{W} \to T$ we make use of resolutions of the ideal $I(\mathcal{X}) \subset R$ of $\mathcal{X}$, the ideal $I(\mathcal{X}_{t_0})$ of $\mathcal{X}_{t_0}$ in the localization $R_{t_0}$ of $R$ at $t_0$, and also of the ideal $J$ of $\mathcal{W} \subset \mathbb{P}^r_A$. Remark that due to {\rm (\ref{Sec3ResolIdSheafX})}, the ideal $I(\mathcal{X}_{t_0})$ has a presentation
\begin{equation}\label{Sec3ResolIdealX0}
  \begin{tikzcd}[row sep=0pt, cells={anchor=west}]
  P_2 \ar[r] & P_1 \ar[r] & I(\mathcal{X}_{t_0}) \ar[r] & 0 \, ,
\end{tikzcd}
\end{equation}
where
{\small $P_1 = \bigoplus\limits^{\beta_1}_{j=1} R_{t_0} (-2) \oplus \bigoplus\limits^{r-1}_1 R_{t_0} (-4)$}
and
{\small $P_2 = \bigoplus\limits^{\beta_2}_{j=1} R_{t_0} (-3) \oplus \bigoplus\limits^{\binom{r-1}{2} + \beta_1}_1 R_{t_0} (-5)$}. By the result of Catalisano and Gimigliano explained in section \ref{Sec2}, the zero locus of the degree two generators of $I(\mathcal{X}_{t_0})$ is precisely the cone $F$ containing the curve $\mathcal{X}_{t_0}$, that is, $V (J_{t_0}) \equiv F$, where $J_{t_0}$ is the ideal of the fiber $\mathcal{W}_{t_0}$ of $\mathcal{W}$ at the point $t_0$. Just like in the proof of \cite[Proposition 1.6]{Cil87} it is obtained that there is a commutative diagram

\begin{equation}\label{Sec3CommuDiagramX0W0}
\begin{tikzcd}[row sep=25pt, cells={anchor=west}]
  Q_2 \ar[r, "{\delta}"] \ar[d, hook] & Q_1 \ar[r] \ar[d, hook] & I(\mathcal{W}_{t_0}) \ar[r] \ar[d, hook] & 0 \\
  P_2 \ar[r, "{\theta}"] & P_1 \ar[r] & I(\mathcal{X}_{t_0}) \ar[r] & 0
\end{tikzcd}
\end{equation}

\noindent with exact rows, where
$Q_1 = \bigoplus\limits^{\beta_1}_{j=1} R_{t_0} (-2)$
and
$Q_2 = \bigoplus\limits^{\beta_2}_{j=1} R_{t_0} (-3)$, $\delta$ is represented by a $\beta_1 \times \beta_2$ matrix of linear forms and $\theta$ is represented by a matrix $M$ of the form
\[
 M = \begin{pmatrix}
  M_{1,1} & M_{1,2} \\
  0 & M_{2,2}
 \end{pmatrix} \, ,
\]
for which $M_{1,1}$ is a $\beta_1 \times \beta_2$ matrix of linear forms, $M_{1,2}$ is a $\beta_1 \times \left( \binom{r-1}{2} + \beta_1 \right)$ matrix and $M_{2,2}$ is a $(r-1) \times \left( \binom{r-1}{2} + \beta_1 \right)$ matrix. As it is explained in \cite[Proposition 1.6]{Cil87}, because of the very flatness of the family $p_{\mathcal{X}} : \mathcal{X} \to T$ there is a presentation
\begin{equation}\label{Sec3ResolIdealX}
  \begin{tikzcd}[row sep=0pt, cells={anchor=west}]
  \mathscr{P}_2 \ar[r, "{\Theta}"] & \mathscr{P}_1 \ar[r] & I(\mathcal{X}) \ar[r] & 0 \, ,
\end{tikzcd}
\end{equation}
of $I(\mathcal{X})$ by free $R$-modules such that the localization of the sequence {\rm (\ref{Sec3ResolIdealX})} at $k(t_0):= A/m_{t_0}$ gives {\rm (\ref{Sec3ResolIdealX0})}, where $m_{t_0} \subset A$ is the ideal corresponding to the point $t_0 \in T$ and $\Theta$ is a homogeneous map represented by a matrix $\mathscr{M}$ of the form
\[
 \mathscr{M} = \begin{pmatrix}
  \mathscr{M}_{1,1} & \mathscr{M}_{1,2} \\
  \mathscr{M}_{2,1} & \mathscr{M}_{2,2}
 \end{pmatrix} \, ,
\]
which modulo the ideal $m_{t_0}$ of $t_0$ gives the matrix $M$. The same degree reasoning argument as in \cite[Proposition 1.6]{Cil87} gives that $\mathscr{M}_{2,1} = 0$ and one can ``chop-off'' from {\rm (\ref{Sec3ResolIdealX})} an exact sequence
\begin{equation}\label{Sec3ResolIdealW}
  \begin{tikzcd}[row sep=0pt, cells={anchor=west}]
  \mathscr{Q}_2 \ar[r, "{\Delta}"] & \mathscr{Q}_1 \ar[r] & J \ar[r] & 0 \, ,
\end{tikzcd}
\end{equation}
where $\Delta$ is homogeneous map, such that tensoring {\rm (\ref{Sec3ResolIdealW})} with $k(t_0)$ we get the first row of {\rm (\ref{Sec3CommuDiagramX0W0})}. This means that the corank of the map $\Delta$ at each localization at $k(t) = A/m_t$, where $m_t$ is the ideal corresponding to $t \in T$, is same for all $t$, or equivalently, that $\dim (J_t)_d$ is same for all $t \in T$. This implies that the family $p_{\mathcal{W}} : \mathcal{W} \to T$ is (very) flat. In particular, it is a family of surfaces in $\mathbb{P}^r$ one of whose fibers, namely $\mathcal{W}_{t_0}$, is a cone over a smooth integral projectively normal curve in $\mathbb{P}^{r-1}$ of genus $\gamma \geq 3$ and degree $e \geq 4\gamma + 5$.

For the remaining part of the proof of the lemma we refer to \cite[Proposition 4.1, p.176--178]{CLM96}. It is proven there that if $p_{\mathcal{W}} : \mathcal{W} \to T$ is a flat family of surfaces in $\mathbb{P}^r$, one of whose fibers is a cone like $\mathcal{W}_{t_0}$ above, then the remaining fibers of the family are also cones over smooth curves of the same genus and degree in $\mathbb{P}^{r-1}$. We remark that the proof uses a result of Pinkham, namely \cite[Theorem 7.5, p.45]{Pin74Ast} about cones over curves of genus $\gamma$ and degree $e$, in which it is required that $e \geq 4\gamma + 5$. Thus the lemma is proved.\end{proof}

This completes Step III and thus the proof of the {\rm Main Theorem}.

\begin{remark}\label{Sec3Remark2-3case}
The technique used in {\rm Step III} of the proof can not be applied to prove \cite[Theorem B]{CIK22}. In that paper we considered a family of curves on cones such that each curve was a double cover of the base and also passed through the vertex on the cone containing it. Just as here, {\rm Proposition \ref{Sec2PropCatGim}} could be applied to obtain a resolution of the ideal of a curve from the family, however, the ideal is generated by polynomials of degree two and three,  which is insufficient to deduce the existence of a presentation like {\rm (\ref{Sec3ResolIdealW})} of the ideal of a similarly defined variety like $\mathcal{W}$ here. That is, one couldn't conclude that $\mathscr{M}_{2,1} = 0$, like we were able to do here due to $I(\mathcal{X}_t)$ being generated by polynomials of degree two and four. In a sense, our present work grew-out from our failure to apply the technique introduced by Ciliberto in \cite{Cil87} and used in \cite{CLM96} to the proof of \cite[Theorem B, Step III]{CIK22}, where we needed to use different arguments.
\end{remark}

\begin{remark}
For a component $\mathcal{D}$ of the Hilbert scheme $\mathcal{I}_{d,g,r}$ the difference $\sigma(\mathcal{D}) := \dim \mathcal{D} -\lambda_{d,g,r}$ is called \emph{superabundance}. It is not difficult to compute about our $\mathcal{H} \subset \mathcal{I}_{3e+1, 3e+3\gamma, e-\gamma+1}$ that $\sigma (\mathcal{H}) = (r-4)e+2(r-5)(e-r)-3$, and using the numerical assumptions in our {\rm Main Theorem}, $\sigma (\mathcal{H}) \geq 224$.
\end{remark}

\end{document}